\documentclass[dvipdfmx]{amsart}
\usepackage[all,cmtip,arrow]{xy}
\usepackage{pb-diagram,pb-xy,
amsxtra,amsmath,amsthm,amssymb,amscd,ascmac,amsfonts,
multicol,tabularx,latexsym,mathrsfs}
\usepackage{amsfonts}
\theoremstyle{definition}

\newcommand{\Q}{\mathbb Q}
\newcommand{\Z}{\mathbb Z}

\newcommand{\Gal}{\mathrm{Gal}}
\newtheorem{thm}{Theorem}[section]

\newtheorem{prop}[thm]{Proposition}
\newtheorem{rem}[thm]{Remark}

\newtheorem{cor}[thm]{Corollary}

\newcommand{\Tor}{\mathrm{Tor}}

\newcommand{\Cl}{\mathrm{Cl}}
\newdimen\minCDarrowwidth
\usepackage{color}

\newcommand{\cE}{\mathcal{E}}
\newcommand{\cG}{\mathcal{G}}

\newcommand{\Ext}{\mathrm{Ext}}
\makeatletter
\@addtoreset{equation}{section}

\makeatother

\date{}
\title[]
{Construction of a $p$-extension of number fields whose unit group has prescribed Galois module structure}

\thanks{
The first author is supported by JSPS KAKENHI Grant Number 22K13898.
}

\author{Takenori Kataoka}
\address{Department of Mathematics, Faculty of Science Division II, Tokyo University of Science.
1-3 Kagurazaka, Shinjuku-ku, Tokyo 162-8601, Japan}
\email{tkataoka@rs.tus.ac.jp}

\author{Manabu Ozaki}
\address{Department of Mathematics, School of Fundamental Science and Engineering, Waseda University.
3-4-1, Ohkubo, Shinjuku-ku, Tokyo 169-8555, Japan}
\email{ozaki@waseda.jp}


\begin{document}

\maketitle

\begin{abstract}
Let $G$ be a finite $p$-group.
We construct a $G$-extension $K/k$ of number fields such that the $p$-adic completion of the unit group of $K$ has a prescribed $\Z_p[G]$-module structure, up to free direct summands.
\end{abstract}

\section{Introduction}\label{sec:intro}
Let $p$ be a fixed prime number.
For a number field $K$, let $E_K$ denote the unit group of $K$ and set $\cE_{K, p} = E_K \otimes_{\Z} \Z_p$.
We are interested in the Galois module structure of $\cE_{K, p}$ when $K$ is a Galois extension of a subfield $k$.
More precisely, the main question in this paper is what kind of $\Z_p[G]$-modules occur as $\cE_{K, p}$ for $G$-extensions $K/k$ when we fix a finite group $G$.

The main theorem of this paper is as follows.

\begin{thm}\label{thm:main}
Let $p$ be an odd prime number and 
$G$ a finite $p$-group. 
Let $C$ be a $\Z_p[G]$-lattice, namely, a finitely generated $\Z_p[G]$-module that is $\Z_p$-free.
We assume that
$
 (C\otimes_{\Z_p} \Q_p) \oplus \Q_p
$
is a free $\Q_p[G]$-module.
Then there exists an unramified $G$-extension $K/k$ of number fields such that 
\[
\cE_{K, p} \simeq C \oplus\Z_p[G]^{\oplus s}
\]
as $\Z_p[G]$-modules for a certain $s\ge 0$.
\end{thm}   

\begin{rem}
\begin{itemize}
\item[(1)]
The assumption on the $\Z_p[G]$-lattice $C$ is necessary for the conclusion; 
this follows from Herbrand's unit theorem (or Remark \ref{rem:nec_cond} below).
\item[(2)]
When $G$ is a cyclic $p$-group, this theorem specializes to the main theorem of a preprint of the second author \cite{Oza_pre}.
In \cite{Oza_pre}, we used Yakovrev's theory, which is especially applicable to cyclic groups.
In this paper, we remove the cyclicity assumption by using a different method.
\item[(3)]
As reviewed in the introduction of \cite{Oza_pre}, earlier works by Burns \cite{B}, Burns-Lim-Maire \cite{BLM}, and 
Kumon-Lim \cite{KL} obtained restrictions on the Galois module structure of $\cE_{K, p}$ and, more generally, on the $S$-unit group variants for finite sets $S$ of primes of $K$.
These restrictions depend on the number of ramified primes in $K/k$, the cardinality of $S$, and the $p$-rank of the $S$-ideal class groups of $K(\mu_p)$ (or of some intermediate fields of $K/k$).
Our theorem shows that no such restriction exists if one does not impose bounds on the ideal class groups.
\item[(4)]
By the Krull-Schmidt theorem, each $\Z_p[G]$-lattice decomposes into a direct sum of indecomposable $\Z_p[G]$-lattices and the multiplicity of each indecomposable component is uniquely determined.
For a finite $p$-group $G$, unless $G$ is cyclic of order $\leq p^2$, it is known that there are infinitely many isomorphism classes of indecomposable $\Z_p[G]$-lattices (Heller-Reiner\cite{HR}).
Therefore, there are many possible choices of $C$.
\end{itemize}
\end{rem}

Let us briefly outline the strategy of the proof.
One of the key ingredients is the Tate sequence due to Ritter--Weiss, which we review in Section \ref{sec:RW}.
It involves the unit group in an explicit way, and also the class group in an implicit way.
Then in Section \ref{sec:proof} we prove the main theorem.
We apply another key ingredient, an earlier result of the second author \cite{Oz} on the existence of prescribed finite $p$-class field tower groups, to find an appropriate extension $K/k$ with prescribed class group.
We then apply Schanuel's lemma to conclude that the condition on the class group suffices to control the unit group.

In Section \ref{sec:variants}, we will also show some variants of the main theorem.

\section{The work of Ritter--Weiss}\label{sec:RW}

In this section, we review the Tate sequence constructed by Ritter--Weiss \cite{RW}.

\subsection{Translation functor \`{a} la Gruenberg}\label{ss:Gru}

We begin with a brief review of the translation functor due to Gruenberg \cite[Section 10.5]{Gru} (see Ritter--Weiss \cite[page 154, Proposition 1]{RW}).

Let $G$ be a finite group and $A$ a $\Z$-module equipped with a $G$-action.
We consider an extension of $G$ by $A$, described by
\[
(\varepsilon) \quad
1 \longrightarrow A \longrightarrow \cG \longrightarrow G \longrightarrow 1,
\]
such that the original action of $G$ on $A$ coincides with the action induced by $(\varepsilon)$ via the conjugation.
We write $I_{\Z[G]} \subseteq \Z[G]$ and $I_{\Z[\cG]} \subseteq \Z[\cG]$ for the augmentation ideals.
Given $(\varepsilon)$, we define
\[
M_{(\varepsilon)} := \Z \otimes_{\Z[A]} I_{\Z[\cG]}.
\]
Then taking $\Z \otimes_{\Z[A]} (-)$ to the exact sequence $0 \to I_{\Z[\cG]} \to \Z[\cG] \to \Z \to 0$ induces an exact sequence
\[
0 \longrightarrow \Tor_1^{\Z[A]}(\Z, \Z) \longrightarrow M_{(\varepsilon)} \longrightarrow \Z[G] \longrightarrow \Z \longrightarrow 0.
\]
Using $\Tor_1^{\Z[A]}(\Z, \Z) \simeq H_1(A, \Z) \simeq A$, we obtain an exact sequence
\[
(\varepsilon^*) \quad
0 \longrightarrow A \longrightarrow M_{(\varepsilon)} \longrightarrow I_{\Z[G]} \longrightarrow 0.
\]

Conversely, given an exact sequence of $\Z[G]$-modules
\[
(\eta) \quad
0 \longrightarrow A \longrightarrow B \longrightarrow I_{\Z[G]} \longrightarrow 0,
\]
we can construct an extension $(\varepsilon)$ such that $(\varepsilon^*)$ is isomorphic to $(\eta)$.
Indeed, the set of isomorphism classes of $(\varepsilon)$ and $(\eta)$ are naturally identified with $H^2(G, A)$ and $\Ext^1_{\Z[G]}(I_{\Z[G]}, A)$ respectively.
We have canonical isomorphisms
\[
H^2(G, A)
\simeq \Ext^2_{\Z[G]}(\Z, A)
\simeq \Ext^1_{\Z[G]}(I_{\Z[G]}, A)
\]
and this composite map is realized by $(\varepsilon) \mapsto (\varepsilon^*)$.

For a prime number $p$, this argument is also valid for $\Z_p$-modules instead of $\Z$-modules as we shall describe quickly.
Let $G$ be a finite group and $A$ a $\Z_p$-module equipped with a $G$-action.
Given an extension
\[
(\varepsilon) \quad
1 \longrightarrow A \longrightarrow \cG \longrightarrow G \longrightarrow 1,
\]
letting $I_{\Z_p[G]} \subseteq \Z_p[G]$ and $I_{\Z_p[\cG]} \subseteq \Z_p[\cG]$ be the augmentation ideals, we define
\[
{}_p M_{(\varepsilon)} := \Z_p \otimes_{\Z_p[A]} I_{\Z_p[\cG]}.
\]
Then we obtain an exact sequence of $\Z_p[G]$-modules
\[
{}_p(\varepsilon^*) \quad
0 \longrightarrow A \longrightarrow {}_pM_{(\varepsilon)} \longrightarrow I_{\Z_p[G]} \longrightarrow 0.
\]
Conversely, each exact sequence of $\Z_p[G]$-modules
\[
0 \longrightarrow A \longrightarrow B \longrightarrow I_{\Z_p[G]} \longrightarrow 0
\]
is obtained as ${}_p(\varepsilon^*)$ by an extension $(\varepsilon)$.

\subsection{Tate sequence of global units}\label{ss:Tate}

We review the construction of the Tate sequence due to Ritter--Weiss, only in the case of unramified extensions, which we actually need.

Let $K/k$ be an unramified Galois extension of number fields (we assume that the archimedean primes are also unramified).
Let $H_K$ be the maximal unramified abelian extension of $K$.
Put $G:=\Gal(K/k),\ \mathcal{G}:=\Gal(H_K/k)$, and
$A:=\Gal(H_K/K)$.
Note that $A$ is isomorphic to the class group $\Cl_K$ of $K$ via the Artin map.
Then we have the natural group extension 
\[
(\varepsilon_{K/k}) \quad
1\longrightarrow A\longrightarrow
\mathcal{G}\longrightarrow G\longrightarrow 1.
\]
Applying the translation functor, we obtain an associated exact sequence
\[
(\varepsilon_{K/k}^*) \quad
 0\longrightarrow A\longrightarrow M_{(\varepsilon_{K/k})}
\longrightarrow I_{\Z[G]}\longrightarrow 0
\]
of $\Z[G]$-modules.

\begin{prop}\label{FE}
Let $K/k$ be an unramified Galois extension of number fields
with $G=\Gal(K/k)$.
Then there exists an exact sequence of finitely generated $\Z[G]$-modules
\[
0\longrightarrow E_K\longrightarrow P\longrightarrow
F\longrightarrow M_{(\varepsilon_{K/k})}\longrightarrow 0
\]
such that $P$ is $G$-cohomologically trivial and $F$ is free.
\end{prop}

\begin{proof}
We will apply the construction of
the Tate sequence by Ritter--Weiss \cite{RW}.

Let $S$ be the set of the archimedean primes of $K$ and
$S'$ a Galois stable finite set of primes of $K$
which is ``larger" in the sense of \cite[p.148]{RW}.
Applying \cite[Theorem 1]{RW} for these $S$ and $S'$, we obtain the exact sequence of $\Z[G]$-modules
\[
0\longrightarrow E_K\longrightarrow A_{\theta}\longrightarrow R_{S'}
\overset{s}{\longrightarrow}\mathrm{Cl}_K\longrightarrow 0,
\]
where
$A_\theta$ is a certain finitely generated cohomologically trivial $\Z[G]$-module and $R_{S'}$
is a $\Z[G]$-lattice defined as the kernel of the canonical
surjective homomorphism
\[
W_{S'}\longrightarrow I_{\Z[G]}
\]
given in \cite[p.149]{RW}.
Then we get the exact sequence
\begin{equation*}
0\longrightarrow E_K\longrightarrow A_\theta
\longrightarrow\mathrm{ker}(s)\longrightarrow 0.
\end{equation*}

On the other hand,
we get the commutative diagram with exact rows
\[
\xymatrix{
	0 \ar[r] & R_{S'} \ar[r] \ar@{->>}[d]_s & W_{S'} \ar[r] \ar[d]_{\tilde{\sigma}} & I_{\Z[G]} \ar[r] \ar@{=}[d] & 0\\
	0 \ar[r] & \Gal(H_K/K) \ar[r] & M_{(\varepsilon_{K/k})} \ar[r] & I_{\Z[G]} \ar[r] & 0
}
\]
by \cite[Theorem 5]{RW}, where we identify $\Cl_K$ with
$\Gal(H_K/K)$ by the Artin map
(we note that $H$ in \cite[p.152]{RW} is equal to
our $M_{(\varepsilon_{K/k})}$).
From the above diagram, we derive the exact sequence
\begin{equation*}
0\longrightarrow\mathrm{ker}(s)\longrightarrow W_{S'}
\longrightarrow M_{(\varepsilon_{K/k})}\longrightarrow 0.
\end{equation*}
Combining these two short exact sequences, we obtain
the exact sequence
\begin{equation*}
0\longrightarrow E_K\longrightarrow A_\theta
\longrightarrow W_{S'}\longrightarrow M_{(\varepsilon_{K/k})}
\longrightarrow 0.
\end{equation*}

The last sequence is what we need.
To see this, we only have to see that $W_{S'}$ is a free $\Z[G]$-module.
By the definition in \cite[p.149]{RW}, we have
\[
W_{S'}=\bigoplus_{\mathfrak{p}\in S'_*-S}\mathrm{Ind}_{G_\mathfrak{p}}^GW_\mathfrak{p}.
\]
Here, $S'_*\subseteq S'$ is a set of representatives of
the $G$-orbits of $S'$ and 
$G_{\mathfrak{p}}$ is the decomposition subgroup of $G$ for $\mathfrak{p}$.
We also have 
$W_{\mathfrak{p}}\simeq\Z[G_{\mathfrak{p}}]$ for each $\mathfrak{p}\in S'-S$
by \cite[Lemma 5(a)]{RW}
since $K/k$ is unramified.
Hence we see that $W_{S'}$ 
is a free $\Z[G]$-module, as claimed.
\end{proof}

In fact, what we need is the following $p$-adic variant.
Let $p$ be a prime number and $K/k$ an unramified Galois extension of number fields
with $G=\Gal(K/k)$.
Let $H_K(p)$ be the maximal unramified abelian $p$-extension field of $K$.
Put $\cG(p) = \Gal(H_K(p)/k)$ and $A(p) = \Gal(H_K(p)/K)$.
Note that $A(p)$ is isomorphic to the $p$-part of the class group $\Cl_K$ of $K$.
Then we have a natural group extension
\[
(\varepsilon_{K/k,\,p}) \quad
1\longrightarrow A(p) \longrightarrow
\cG(p) \longrightarrow G \longrightarrow 1.
\]
Applying the translation functor, we obtain an associated exact sequence
\[
{}_p(\varepsilon_{K/k, \, p}^*) \quad
0\longrightarrow A(p) \longrightarrow {}_pM_{(\varepsilon_{K/k, \, p})}
\longrightarrow I_{\Z_p[G]}\longrightarrow 0
\]
of $\Z_p[G]$-modules.

\begin{cor}\label{cor:Tate_p}
Let $p$ be a prime number and $K/k$ an unramified Galois extension of number fields
with $G=\Gal(K/k)$.
Then there exists an exact sequence of finitely generated $\Z_p[G]$-modules
\[
0\longrightarrow \mathcal{E}_{K,p}\longrightarrow P_p\longrightarrow
F_p\longrightarrow {}_pM_{(\varepsilon_{K/k,\,p})}
\longrightarrow 0
\]
such that $P_p$ is $G$-cohomologically trivial and $F_p$ is free.
Moreover, if $\mu_p\not\subseteq K$, then $P_p$ is projective.
\end{cor}

\begin{proof}
Taking the tensor product of the exact sequence
given in Proposition \ref{FE} with $\Z_p$,
we get the exact sequence of $\Z_p[G]$-modules
\[
0
\longrightarrow \mathcal{E}_{K,p}\longrightarrow P\otimes_{\Z}\Z_p
\longrightarrow
F\otimes_{\Z}\Z_p\longrightarrow M_{(\varepsilon_{K/k})}\otimes_{\Z}\Z_p
\longrightarrow 0.
\]
By construction, we have an isomorphism $M_{(\varepsilon_{K/k})}\otimes_{\Z}\Z_p
\simeq {}_pM_{(\varepsilon_{K/k,\,p})}$.
Therefore, setting $P_p = P\otimes_{\Z}\Z_p$ and $F_p = F\otimes_{\Z}\Z_p$, we obtain the sequence.
It remains to mention that, if $\mu_p \not \subseteq K$, then $P_p$ is projective over $\Z_p[G]$.
This is because $P_p$ is then $\Z_p$-torsion-free and $G$-cohomologically trivial.
\end{proof}

\begin{rem}\label{rem:nec_cond}
In Corollary \ref{cor:Tate_p}, let us assume $\mu_p \not \subseteq K$ and $K/k$ is a $p$-extension.
Then $\Z_p[G]$ is a local ring, so the projective module $P_p$ must be free.
By exact sequence ${}_p(\varepsilon_{K/k, \, p}^*)$, we see that $({}_pM_{(\varepsilon_{K/k,\,p})} \otimes_{\Z_p} \Q_p) \oplus \Q_p \simeq \Q_p[G]$.
Therefore, Corollary \ref{cor:Tate_p} implies that $(\cE_{K, p} \otimes_{\Z_p} \Q_p) \oplus \Q_p$ is a free $\Q_p[G]$-module.
\end{rem}

\section{Proof of the main theorem}\label{sec:proof}

Let $p$ be a prime number and $G$ a finite $p$-group.
To prove the main theorem, considering Corollary \ref{cor:Tate_p} and sequence ${}_p(\varepsilon_{K/k, p}^*)$, we first show the following.

\begin{prop}\label{PP}
Let $C$ be a $\Z_p[G]$-lattice such that $(C\otimes_{\Z_p} \Q_p) \oplus \Q_p$ is a free $\Q_p[G]$-module.
Then there exists a $\Z_p[G]$-module $B$ that fits into exact sequences of finitely generated $\Z_p[G]$-modules
\begin{equation}\label{eFF}
0\longrightarrow C\longrightarrow F_1
\longrightarrow F_2
\longrightarrow B
\longrightarrow 0,
\end{equation}
where $F_1$ and $F_2$ are free, and
\begin{equation}\label{eM}
0\longrightarrow A\longrightarrow B
\longrightarrow I_{\Z_p[G]}\longrightarrow 0,
\end{equation}
where $A$ is finite.
\end{prop}

\begin{proof}
We first show the existence of an exact sequence of $\Z_p[G]$-lattices
\begin{equation}\label{e1}
0\longrightarrow C\longrightarrow F_1\longrightarrow N\longrightarrow 0
\end{equation}
such that $F_1$ is free.
For this, we use the notion of the dual lattices:
For any $\Z_p[G]$-lattice $L$, we define the dual $\Z_p[G]$-lattice as 
$L^*:=\mathrm{Hom}_{\Z_p}(L,\Z_p)$
on which $G$ acts by $\sigma f:=f\circ\sigma^{-1}$ for $\sigma\in G$ and $f\in L^*$.
Then we have a natural isomorphism $L\simeq(L^*)^*$, and $L$ is free if and only if so is $L^*$.
Therefore, to construct \eqref{e1}, we only have to construct an exact sequence $0 \to N^* \to F_1^* \to C^* \to 0$, which is quite standard.

On the other hand, we can take an exact sequence
\begin{equation}\label{e2}
0\longrightarrow L \longrightarrow F_2
\longrightarrow I_{\Z_p[G]}\longrightarrow 0
\end{equation}
such that $F_2$ is free of finite rank.

It follows from the assumption on $C$ and sequences \eqref{e1} and \eqref{e2}
that
\[
N\otimes_{\Z_p}\Q_p\simeq \Q_p\oplus\Q_p[G]^{\oplus s},\ \ 
L\otimes_{\Z_p}\Q_p\simeq
\Q_p\oplus\Q_p[G]^{\oplus t}
\]
for certain $s,t\ge 0$.
By adding free direct factors to
$F_1$ and $N$ of \eqref{e1}, or to $L$ and $F_2$ of \eqref{e2} if necessary,
we may assume that $s=t$, that is,
\[
N\otimes_{\Z_p}\Q_p \simeq L \otimes_{\Z_p}\Q_p.
\]
It follows that there exists an injection $N \hookrightarrow L$ with finite cokernel $A$, so we have an exact sequence
\[
0 \longrightarrow N \longrightarrow L \longrightarrow A \longrightarrow 0.
\]
By combining this with \eqref{e1}, we obtain an exact sequence
\[
0 \longrightarrow C \longrightarrow F_1 \longrightarrow L \longrightarrow A \longrightarrow 0.
\]
Using this sequence and \eqref{e2}, and letting $B$ be the cokernel of the composite homomorphism $F_1 \to L \hookrightarrow F_2$, we obtain the desired exact sequences \eqref{eFF} and \eqref{eM}.
\end{proof}

We now construct an extension $K/k$.

\begin{prop}\label{prop:ur_ext}
Let $B$ be a finitely generated $\Z_p[G]$-module fitting into \eqref{eM} with $A$ finite.
Then there exists an unramified $G$-extension of number fields
$K/k$ such that
\[
B\simeq {}_p M_{(\varepsilon_{K/k,p})}
\]
as $\Z_p[G]$-modules.
Moreover, we can choose $K/k$ so that
$[k:\Q]$ is arbitrarily large, and that $\mu_p\not\subseteq K$
if $p \geq 3$.
\end{prop}
\begin{proof}
As explained in Subsection \ref{ss:Gru}, there exists a group extension
\[
(\varepsilon) \quad
1 \longrightarrow A \longrightarrow \cG \longrightarrow G \longrightarrow 1
\]
of $G$ by $A$ such that ${}_p (\varepsilon^*)$ is isomorphic to the given sequence \eqref{eM}.
In particular, we have ${}_pM_{(\varepsilon)} \simeq B$.

By applying \cite{Oz}, which states that any finite $p$-group can be realized as the Galois group of a $p$-class field tower,
we find a number field $k$ with an isomorphism
\[
\mathcal{G}\simeq\Gal(L_p(k)/k),
\]
where $L_p(k)$ is the maximal unramified $p$-extension of $k$. 
Here we may choose $k$ so that $[k:\Q]$ is arbitrarily large, and that $\mu_p\not\subseteq L_p(k)$ if $p \geq 3$ (See \cite{HMR}, \cite{Oz}).

Let $K$ be the intermediate field of $L_p(k)/k$ associated to the subgroup $A$ of $\cG$.
Then $K/k$ is an unramified $G$-extension.
Moreover, we have $H_K(p) = L_p(k)$, so the group extension $(\varepsilon_{K/k,\,p})$ is identified with $(\varepsilon)$.
Hence we have
\[
{}_pM_{(\varepsilon_{K/k,\,p})}\simeq
{}_pM_{(\varepsilon)}
\]
as $\Z_p[G]$-modules.
This completes the proof.
\end{proof}

We are now ready to prove the main theorem.

\begin{proof}[Proof of Theorem \ref{thm:main}]
We suppose that $p$ is odd.
Let $C$ be a $\Z_p[G]$-lattice as in the statement.
By Proposition \ref{PP}, we can construct a module $B$ satisfying exact sequences \eqref{eFF}
and \eqref{eM}.
Then by Proposition \ref{prop:ur_ext}, we can construct an unramified $G$-extension $K/k$ such that $B \simeq {}_pM_{(\varepsilon_{K/k,p})}$.
We can also impose the conditions $\mu_p \not \subseteq K$ and that the $\Z_p$-rank of $\mathcal{E}_{K,p}$
is greater than that of $C$.
For such an extension $K/k$, by Corollary \ref{cor:Tate_p}, we have an sequence 
\[
0\longrightarrow \mathcal{E}_{K,p}
\longrightarrow F_3
\longrightarrow F_4
\longrightarrow {}_pM_{(\varepsilon_{K/k,p})}
\longrightarrow 0
\]
such that
$F_3 = P_p$ and $F_4 = F_p$ are both free.

Since $B \simeq {}_pM_{(\varepsilon_{K/k,p})}$ and $F_1, F_2, F_3, F_4$ are all free, we can now apply Schanuel's lemma to this sequence and \eqref{eFF}.
As a result, we get an isomorphism
\[
\mathcal{E}_{K, p}\oplus F_1\oplus F_4\simeq C\oplus F_3\oplus F_2.
\]
This shows that $C$ and $\mathcal{E}_{K, p}$ are isomorphic up to free direct summands.
Moreover, by the Krull-Schmidt theorem for finitely generated $\Z_p[G]$-modules, we can cancel the free modules.
By the condition on the $\Z_p$-ranks of $\cE_{K, p}$ and $C$, 
the summands on the left hand side are all cancelled, and this completes the proof of Theorem \ref{thm:main}.
\end{proof}

\begin{rem}\label{rem:final}
In Theorem \ref{thm:main}, there exist both totally real $k$ and totally imaginary $k$ satisfying the conclusion.
To show this, we only have to apply the more detailed statement of \cite{HMR} when constructing $K/k$ in Proposition \ref{prop:ur_ext}.
\end{rem}

\section{Variants}\label{sec:variants}

In this section, we show two variants of Theorem \ref{thm:main} by using the same strategy.

To begin with, we observe a generalization of Proposition \ref{FE}.
Let $K/k$ be an unramified extension and $S$ a set of places of $k$ containing the archimedean primes; Proposition \ref{FE} dealt with the case where $S$ does not contain a non-archimedean prime.
We assume that all places in $S$ totally split in $K/k$.
Let $E_{K, S}$ be the $S$-unit group of $K$.
Let $H_{K, S}$ be the maximal intermediate field of $H_K/K$ such that every prime lying above a prime in $S$ totally splits, so $\Gal(H_{K, S}/K)$ is isomorphic to the $S$-ideal class group $\Cl_{K, S}$.
We have a natural group extension
\[
(\varepsilon_{K/k, S}) \quad
1\longrightarrow \Cl_{K, S} \longrightarrow
\Gal(H_{K, S}/k) \longrightarrow G \longrightarrow 1,
\]
which induces an exact sequence
\[
(\varepsilon_{K/k, S}^*) \quad
 0\longrightarrow \Cl_{K, S} \longrightarrow M_{(\varepsilon_{K/k, S})}
\longrightarrow I_{\Z[G]}\longrightarrow 0
\]
of $\Z[G]$-modules.

In this case, as a generalization of Proposition \ref{FE}, there exists an exact sequence of finitely generated $\Z[G]$-modules
\[
0\longrightarrow E_{K, S}\longrightarrow P\longrightarrow
F\longrightarrow M_{(\varepsilon_{K/k, S})}\longrightarrow 0
\]
such that $P$ is $G$-cohomologically trivial and $F$ is free.
To generalize the proof, we only have to notice that the module $W_{S'}$ in the proof is still free over $\Z[G]$ because of the assumption that all places in $S$ totally split in $K/k$.

\subsection{First variant}

As a first variant of Theorem \ref{thm:main}, we replace $E_K$ by $E_{K, S_p}$, where $S_p$ denotes the set of $p$-adic places:

\begin{prop}
Let $p$ be an odd prime number and 
$G$ a finite $p$-group. 
Let $C$ be a $\Z_p[G]$-lattice such that $(C\otimes_{\Z_p} \Q_p) \oplus \Q_p$ is a free $\Q_p[G]$-module.
Then there exists an unramified $G$-extension $K/k$ of number fields in which all $p$-adic primes totally split such that 
\[
E_{K, S_p} \otimes_{\Z} \Z_p \simeq C \oplus\Z_p[G]^{\oplus s}
\]
as $\Z_p[G]$-modules for a certain $s\ge 0$.
\end{prop}   

\begin{proof}
By \cite{HMR} and \cite{Oz}, the extension $L_p(k)/k$ in the proof of Proposition \ref{prop:ur_ext} can be chosen so that all $p$-adic places totally split.
Then we have $\Cl_{K, S_p} = \Cl_K$ and the same proof works.
\end{proof}

\subsection{Second variant}

To derive the second variant, we use the following proposition instead of the key ingredient \cite{Oz}.
For a number field $K$, we write $\mu(K)$ for the set of roots of unity of $K$.

\begin{prop}\label{prop:ur_G}
For any finite group $\cG$, there exists an unramified $\cG$-extension $L/k$ of number
fields such that $\mu(L)=\{1,-1\}$.
\end{prop}

\begin{proof}
Because any finite group can be embedded into a symmetric group $S_n$ for some $n \geq 1$, we may assume from the beginning that $\cG = S_n$.

It follows from \cite[Theorem 2]{Uch} and its proof that
there exist infinitely many quadratic fields $\Q(\sqrt{D})\ (D\in\Z)$
such that $\Q(\sqrt{D})$ is a quadratic subfield 
of a certain $S_n$-extension $L_0/\Q$ and $L_0/\Q(\sqrt{D})$ is unramified
at all the non-archimedean primes.
We take such a $D$ with $\Q(\sqrt{D})\ne\Q(\sqrt{-3}),\,\Q(\sqrt{-1})$.

Let $\delta:=D/|D|$ be the signature of $D$.
We take a prime number $l\ge 5$ such that $l\nmid D$ and $-\delta l\equiv 1\pmod{4}$.
Then $k:=\Q(\sqrt{-\delta lD})$ is an imaginary quadratic field such that $k(\sqrt{D}) = k(\sqrt{-\delta l})$ is unramified over $k$.

Putting $L:=L_0k$, we claim that $L/k$ satisfies the desired properties.
We have $k \cap L_0=\Q$ by considering the ramification of $l$, so $L/k$ is an $S_n$-extension.
The extension $L/k$ is also unramified because so are $L/k(\sqrt{D})$ and $k(\sqrt{D})/k$; since $k$ is totally imaginary, the archimedean primes do not matter.

It remains to show that $\mu(L) = \{1, -1\}$.
Assume that a root of unity $\zeta$ belongs to $L$.
Then $k(\zeta)/k$ is an abelian subextension of the $S_n$-extension $L/k$, so $k(\zeta)$ is contained in its unique quadratic subextension, which is exactly $k(\sqrt{D})$.
Hence we have $k(\zeta) \subseteq k(\sqrt{D})$, that is, $\Q(\zeta) \subseteq k(\sqrt{D})$.
Considering the five subfields of $k(\sqrt{D}) = \Q(\sqrt{-\delta l D}, \sqrt{D})$, and using $l \geq 5$ and $\Q(\sqrt{D})\ne\Q(\sqrt{-3}),\,\Q(\sqrt{-1})$, we can easily conclude that $\zeta \in \{1, -1\}$ holds.
\end{proof}

Now we present the second variant of Theorem \ref{thm:main},
which is a result over $\Z$ rather than $\Z_p$, while we allow the set $S$ to vary.
Recall that two modules are said to be projectively equivalent if they become isomorphic after taking direct sums with finitely generated projective modules.

\begin{prop}\label{prop:2nd_var}
Let $G$ be a finite group of odd order.
Let $C$ be a $\Z[G]$-lattice such that $(C \otimes_{\Z} \Q) \oplus \Q$ is a free $\Q[G]$-module.
Then there exist a $G$-extension $K/k$ of number fields and a finite set $S$ of places of $k$ containing the archimedean primes such that $E_{K, S}/\mu(K)$ and $C$ are projectively equivalent to each other as $\Z[G]$-modules.
\end{prop}

\begin{proof}
In the same way as in Proposition \ref{PP}, we can find a $\Z[G]$-module $B$ that fits into exact sequences
\[
0 \longrightarrow C \longrightarrow F_1 \longrightarrow F_2 \longrightarrow B \longrightarrow 0
\]
and
\[
0 \longrightarrow A \longrightarrow B \longrightarrow I_{\Z[G]} \longrightarrow 0,
\]
where $F_1$ and $F_2$ are free of finite rank and $A$ is a finite module.
The last sequence is induced by a group extension
\[
(\varepsilon) \quad
1 \longrightarrow A \longrightarrow \cG \longrightarrow G \longrightarrow 1.
\]

Let us construct $K/k$ and $S$ such that $(\varepsilon_{K/k, S})$ is isomorphic to this group extension $(\varepsilon)$.
By applying Proposition \ref{prop:ur_G}, we can take an unramified $\cG$-extension $L/k$ of number fields such that $\mu(L) = \{1, -1\}$.
We define $K$ as the intermediate field of $L/k$ associated to the subgroup $A$ of $\cG$.
Then $K/k$ is an unramified $G$-extension and $A$ is a quotient module of the class group $\Cl_K \simeq \Gal(H_K/K)$.
We then take a finite set $S$ of places of $k$ such that the Frobenius automorphisms in $\Gal(H_K/k)$ of the places $v \in S$ generate the subgroup $\Gal(H_K/L)$.
Such an $S$ exists by the Chebotarev density theorem.
Then we have $\Cl_{K, S} \simeq \Gal(L/K) \simeq A$.

Now we have constructed $K/k$ and $S$ such that $M_{(\varepsilon_{K/k, S})} \simeq B$.
The generalization of Proposition \ref{FE} gives us an exact sequence
\[
0\longrightarrow E_{K, S}/\mu(K) \longrightarrow P' \longrightarrow
F\longrightarrow M_{(\varepsilon_{K/k, S})}\longrightarrow 0,
\]
where $P'$ is the quotient of $P$ by the image of $\mu(K)$.
Since the order of $\mu(K)$ is two while the order of $G$ is odd, the module $P'$ is still $G$-cohomologically trivial.
Therefore, $P'$ is a projective $\Z[G]$-module.
By applying Schanuel's lemma, we complete the proof.
\end{proof}

The assumption that the order of $G$ is odd can be removed if we work over $\Z' = \Z[1/2]$ instead of $\Z$ as follows.

\begin{prop}
Let $G$ be a finite group.
Let $C$ be a $\Z'[G]$-lattice such that $(C \otimes_{\Z'} \Q) \oplus \Q$ is a free $\Q[G]$-module.
Then there exist a $G$-extension $K/k$ of number fields and a finite set $S$ of places of $k$ containing the archimedean primes such that $E_{K, S} \otimes_{\Z} \Z'$ and $C$ are projectively equivalent to each other as $\Z'[G]$-modules.
\end{prop}

\begin{proof}
We trace the proof of Proposition \ref{prop:2nd_var}, replacing $\Z$ by $\Z'$ appropreately.
Then as in the final paragraph, we have $K/k$ and $S$ such that $M_{(\varepsilon_{K/k, S})} \otimes_{\Z} \Z' \simeq B$ and the exact sequence in the final paragraph is replaced by
\[
0\longrightarrow E_{K, S} \otimes_{\Z} \Z' \longrightarrow P \otimes_{\Z} \Z' \longrightarrow
F \otimes_{\Z} \Z' \longrightarrow M_{(\varepsilon_{K/k, S})} \otimes_{\Z} \Z' \longrightarrow 0.
\]
Since the order of $\mu(K)$ is two, the $\Z'[G]$-module $E_{K, S} \otimes_{\Z} \Z'$ is torsion-free, and it follows that $P \otimes_{\Z} \Z'$ is a projective $\Z'[G]$-module.
Therefore, we can again apply Schanuel's lemma to complete the proof.
\end{proof}

\end{document}